\documentclass{amsart}

\usepackage[dvipsnames]{color}
\definecolor{green}{rgb}{0.0, 0.5, 0.05}
\definecolor{teal}{rgb}{0.0, 0.5, 0.5}
\definecolor{purple}{rgb}{0.5, 0.0, 0.5}
\definecolor{fuchsia}{rgb}{1.0, 0.0, 1.0}

\usepackage{url}
\usepackage{amssymb}
\usepackage{mathtools}
\usepackage{xypic}
\usepackage{hyperref}
\usepackage{etoolbox}
\usepackage{comment}

\newtheorem{theorem}{Theorem}[section]

\newtheorem{lemma}[theorem]{Lemma}
\newtheorem{corollary}[theorem]{Corollary}
\theoremstyle{definition}

\newtheorem{remark}[theorem]{Remark}
\newtheorem{example}[theorem]{Example}
\theoremstyle{plain}

\numberwithin{equation}{section}

\def\Q{\mathbb Q}

\def\Z{\mathbb Z}

\def\ge{\geqslant}
\def\le{\leqslant}

\begin{document}

\title{Paratrophic Determinants over $\mathbb{Z}/N\mathbb{Z}$ via Discrete Fourier Transform}

\author{Hang LIU}
\address{School of Mathematical Sciences, Shenzhen University, Shenzhen, 518060, Guangdong, P. R. China}
\email{liuhang@szu.edu.cn}

\begin{abstract}
In this note, we investigate the paratrophic determinants attached to the multiplicative semigroup $\mathbb{Z}/N\mathbb{Z}$. 
We show that, via discrete Fourier, cosine, and sine transforms, these determinants factor into products of group determinants indexed by $d \mid N$. 
This yields explicit formulas for several determinant families, including determinants involving periodic Bernoulli functions and powers of the tangent function.
As an application, we also prove a corrected version of a conjecture of Sun Zhi-Wei.
\end{abstract}


\maketitle

\section{Introduction}
Let $N\ge 2$ and $k\ge 1$ be integers. We define the periodic Bernoulli function
\[
\widetilde B_k(x):=B_k(\{x\}),\qquad \{x\}\in[0,1),
\]
where $B_k(x)$ is the $k$-th Bernoulli polynomial with $B_1(x)=x-\frac{1}{2}, B_2(x)=x^2-x+\frac{1}{6}$ and $\{x\}$ denotes the fractional part. 
For $k=1$, define $\widetilde{B}_1(0)=0$ by convention.

Define
$$
B_{k, N}=\left\{\begin{array}{lc}
\left(\widetilde{B}_k(i j / N)\right)_{0 \leqslant i, j \leqslant \lfloor N/2 \rfloor} & (k \text { even}), \\
\left(\widetilde{B}_k(i j / N)\right)_{1 \leqslant i, j \leqslant \lfloor (N-1)/2 \rfloor} & (k \text { odd}) .
\end{array}\right.
$$
For $k=2$, Brunault \cite{B26} proved that $\det(B_{2,N})$ is related to the cardinality of the cuspidal subgroup of the modular curve $X_1(N)$. 
For $k=3$, the matrix $B_{3,N}$ appears in the work on $K_2$ of families of elliptic curves with $N$-torsion points \cite{BDLR24}.

Let $m, n$ be positive integers, $N = 2n+1$ and
$$T_{m,N} = \left(\tan^m\left(\frac{\pi i j}{N}\right)\right)_{1\le i,j \le n}.$$
For $m=1, 2$, Sun Zhi-Wei \cite[Conjecture 5.2]{S24} made some conjectures on the integrality of $\det(T_{m,N})$. If $m=1$ and $N$ is a prime, the conjecture was solved by Guo in \cite{G22}.

These matrices can be studied in the framework of Frobenius paratrophic matrices. We can derive formulas for their determinants by specializing \cite[Theorem 6.14]{S22}.
However, the resulting formulas are quite complicated in practice.

In this note, we show that these determinants can be reduced to products of group determinants indexed by $d \mid N$ using discrete Fourier/cosine/sine transforms. This method has two advantages. First, it has a clear DFT interpretation. Second, in our setting, the entries of the resulting group determinants are well studied. We obtain formulas of a different, simpler shape than those derived from \cite[Theorem 6.14]{S22}.

We remark that Carlitz and Olson \cite{CO55} already studied the discrete sine transform of the matrix 
$
\left(\cot\left(\pi i j/p \right)\right)_{1\le i,j \le \frac{p-1}{2}}
$
where $p$ is an odd prime. They gave the relation between its determinant and Maillet's determinant. 
Our method extends their idea to a broader class of determinants.

The paper is organized as follows. 
In Section 2, we develop a Fourier-analytic framework for paratrophic determinants attached to $\Z/N\Z$ and show that the transformed matrices decompose into blocks given by group determinants. In Section 3, we apply this framework to Bernoulli and tangent matrices, obtaining explicit determinant formulas in Theorem \ref{thm:det_B_formulas} and Theorem \ref{thm:det_T_formulas}, and a corrected version of Sun's conjecture in Corollary \ref{cor:Sun}. These formulas were also numerically verified with PARI/GP scripts available at \cite{L26}.

\section{Paratrophic determinants attached to \(\mathbb Z/N\mathbb Z\)}
We use Steinberg's notion of Frobenius paratrophic matrix and determinant \cite{S22}, but
for the present purposes it is more convenient to work directly with the three
concrete matrix families attached to \(\mathbb Z/N\mathbb Z\). Throughout,
\[
\mathrm{i}:=\sqrt{-1},\qquad
N^+:=\Bigl\lfloor\frac N2\Bigr\rfloor,\qquad
N^-:=\Bigl\lfloor\frac{N-1}{2}\Bigr\rfloor.
\]
For \(a\in \mathbb Z/N\mathbb Z\), we write \(\bar a\in\{0,1,\dots,N-1\}\) for
its standard representative modulo \(N\), and
\[
[a]:=\{\pm a\}\in (\mathbb Z/N\mathbb Z)/\{\pm1\}
\]
for its \(\{\pm1\}\)-orbit.

\subsection{The matrices \(X_N\), \(Y_N\), and \(Z_N\)}

Let \(\{x_a:a\in \mathbb Z/N\mathbb Z\}\) be commuting indeterminates, and define
\[
X_N:=\bigl(x_{\overline{ij}}\bigr)_{0\le i,j\le N-1}.
\]
This is the semigroup matrix of the multiplicative monoid \(\mathbb Z/N\mathbb Z\).

Next let \(\{y_{[a]}:[a]\in (\mathbb Z/N\mathbb Z)/\{\pm1\}\}\) be commuting
indeterminates. Identifying \((\mathbb Z/N\mathbb Z)/\{\pm1\}\) with the set of
representatives \(\{0,1,\dots,N^+\}\), we define
\[
Y_N:=\bigl(y_{[ij]}\bigr)_{0\le i,j\le N^+},\qquad Y_N':=\bigl(y_{[ij]}\bigr)_{1\le i,j\le N^+}.
\]
Here, $Y_N$ is the semigroup matrix of the quotient semigroup
\((\mathbb Z/N\mathbb Z)/\{\pm1\}\) and $Y_N'$ is a submatrix of $Y_N$. We study $Y_N'$ for its application to tangent matrices.

Finally, let \(\{z_r:1\le r\le N^-\}\) be commuting indeterminates. We extend
this family to all residue classes by odd reduction: for \(a\in \mathbb Z/N\mathbb Z\),
set
\[
z_{\langle a\rangle}=
\begin{cases}
0,& a=0 \text{ or } a=N/2 \text{ for $N$ even} ,\\
z_a,& 1\le a\le N^-,\\
-z_{N-a},& N^-<a\le N-1.
\end{cases}
\]
We then define
\[
Z_N:=\bigl(z_{\langle ij\rangle}\bigr)_{1\le i,j\le N^-}.
\]
This is the paratrophic matrix arising from the odd quotient
\[
\mathbb C[\mathbb Z/N\mathbb Z]\Big/\big\langle [a]+[-a]:a\in \mathbb Z/N\mathbb Z\big\rangle.
\]

\subsection{Fourier, cosine and sine transforms}
Let $\omega_N:=e^{2\pi \mathrm{i}/N}$.
We consider the discrete Fourier, cosine, and sine
transform matrices $F_N$, $C_N$, and $S_N$ as follows
\[
F_N:=\bigl(\omega_N^{mn}\bigr)_{0\le m,n\le N-1},
\]
\[
C_N:=\bigl(c_n\cos(2\pi mn/N)\bigr)_{0\le m,n\le N^+},
\qquad
c_n=
\begin{cases}
\frac12,& n=0,\\
\frac12,& N \text{ even and } n=N/2,\\
1,& \text{otherwise},
\end{cases}
\]
and
\[
S_N:=\bigl(\sin(2\pi mn/N)\bigr)_{1\le m,n\le N^-}.
\]

We write
\[
\widehat X_N:=F_NX_N,\qquad
\widehat Y_N:=C_NY_N,\qquad
\widehat Z_N:=S_NZ_N.
\]

For each divisor \(d\mid N\), put
$N_d:=N/d$.
For \(t\in \mathbb Z/N_d\mathbb Z\), define
\begin{align*}
\widehat x_d(t)
&:=\sum_{r=0}^{N_d-1}x_{\overline{dr}}\,\omega_{N_d}^{tr},\\
\widehat y_d(t)
&:=\frac12\,y_{[0]}
+\frac12\,\mathbf 1_{2\mid N_d}\,y_{[N/2]}(-1)^t
+\sum_{r=1}^{N_d^-}y_{[dr]}\cos\!\Bigl(\frac{2\pi tr}{N_d}\Bigr)\\
&=\frac{1}{2}\sum_{r=1}^{N_d}y_{[dr]}\cos\!\Bigl(\frac{2\pi tr}{N_d}\Bigr),\\
\widehat z_d(t)
&:=\sum_{r=1}^{N_d^-}z_{\langle dr \rangle}\sin\!\Bigl(\frac{2\pi tr}{N_d}\Bigr)=\frac{1}{2}\sum_{r=1}^{N_d}z_{\langle dr \rangle}\sin\!\Bigl(\frac{2\pi tr}{N_d}\Bigr).
\end{align*}

The point of these definitions is that, for a fixed column index $k$ with
$d=(k,N)$, the transformed column is supported only on rows $m$ with $d\mid m$,
and its nonzero entries are given by Fourier-type sums on the smaller modulus $N_d$.

\begin{lemma}\label{lemma:structure-hats}
Let $m$ and $k$ be row and column indices of matrices and put \(d:=\gcd(k,N)\).
If $d\nmid m$, then 
$$(\widehat X_N)_{m,k}=0, \quad (\widehat Y_N)_{m,k}=0, \quad (\widehat Z_N)_{m,k}=0.$$
If $d \mid m$, denote $m=du$, $k=dv$, then $\gcd(v,N_d)=1$ and 
\[
(\widehat X_N)_{m,k}=d\,\widehat x_d(uv^{-1}), \quad (\widehat Y_N)_{m,k}=d\,\widehat y_d(uv^{-1}), \quad (\widehat Z_N)_{m,k}=d\,\widehat z_d(uv^{-1}),
\]
where $v^{-1}$ is the inverse of $v$ modulo $N_d$.
\end{lemma}

\begin{proof}
We begin with \(X_N\). By definition,
\begin{equation*}
(\widehat X_N)_{m,k}
=\sum_{j=0}^{N-1}\omega_N^{mj}\,x_{\overline{jk}}.
\end{equation*}
Since \(\overline{jk}\) depends only on \(j\bmod N_d\), we may
write \(j=j_0+tN_d\) with \(0\le j_0<N_d\) and \(0\le t<d\). This gives
\[
(\widehat X_N)_{m,k}
=
\sum_{j_0=0}^{N_d-1}x_{\overline{j_0k}}\omega_N^{mj_0}
\sum_{t=0}^{d-1}e^{2\pi \mathrm{i}mt/d}.
\]
The inner geometric sum is \(0\) unless \(d\mid m\), and equals \(d\) when
\(d\mid m\). This proves the vanishing statement. If now \(m=du\) and \(k=dv\)
with \(v\in(\mathbb Z/N_d\mathbb Z)^\times\), then
\[
(\widehat X_N)_{du,dv}
=
d\sum_{j_0=0}^{N_d-1}x_{\overline{dj_0v}}\omega_{N_d}^{uj_0}.
\]
Since $v$ is invertible modulo $N_d$, the change of variables
\(r\equiv j_0v\pmod N_d\) yields
\[
(\widehat X_N)_{du,dv}
=
d\sum_{r=0}^{N_d-1}x_{\overline{dr}}\omega_{N_d}^{uv^{-1}r}
=
d\,\widehat x_d(uv^{-1}).
\]

For \(Y_N\), pairing $j$ and $N-j$, and using the symmetry \(y_{[a]}=y_{[-a]}\), we have
\[
(\widehat Y_N)_{m,k}
=
\sum_{j=0}^{N^+}c_j\cos\!\Bigl(\frac{2\pi mj}{N}\Bigr)y_{[jk]}
=
\frac12\sum_{j=0}^{N-1}y_{[jk]}\omega_N^{mj}.
\]
The same decomposition \(j=j_0+tN_d\) now gives the same geometric series, and
hence the same vanishing criterion:
\[
(\widehat Y_N)_{m,k}=0\qquad\text{if }d\nmid m.
\]
If \(m=du\) and \(k=dv\), then
\[
(\widehat Y_N)_{du,dv}
=
\frac d2\sum_{r=0}^{N_d-1}y_{[dr]}\omega_{N_d}^{uv^{-1}r}.
\]
Pairing \(r\) with \(N_d-r\), and using \(y_{[d(N_d-r)]}=y_{[dr]}\), we obtain
\[
(\widehat Y_N)_{du,dv}
=
d\left(
\frac12\,y_{[0]}
+\frac12\,\mathbf 1_{2\mid N_d}\,y_{[N/2]}(-1)^{uv^{-1}}
+\sum_{r=1}^{N_d^-}y_{[dr]}\cos\!\Bigl(\frac{2\pi uv^{-1}r}{N_d}\Bigr)
\right),
\]
which is precisely \(d\,\widehat y_d(uv^{-1})\).

The sine case is handled in the same way, now using the oddness of
\(a\mapsto z_{\langle a\rangle}\). Indeed, we have
\[
(\widehat Z_N)_{m,k}
=
\sum_{j=1}^{N^-}\sin\!\Bigl(\frac{2\pi mj}{N}\Bigr)z_{\langle jk\rangle}
=
\frac{1}{2\mathrm{i}}\sum_{j=0}^{N-1}\omega_N^{mj}z_{\langle jk\rangle},
\]
because \(z_{\langle -a\rangle}=-z_{\langle a\rangle}\), and the terms at \(0\)
and, when \(N\) is even, at \(N/2\), vanish. Repeating the above argument shows
that
\[
(\widehat Z_N)_{m,k}=0\qquad\text{if }d\nmid m.
\]
If \(m=du\) and \(k=dv\), then
\[
(\widehat Z_N)_{du,dv}
=
\frac d{2\mathrm{i}}
\sum_{r=0}^{N_d-1}z_{\langle dr\rangle}\omega_{N_d}^{uv^{-1}r}.
\]
Pairing \(r\) with \(N_d-r\), and using
\(z_{\langle d(N_d-r)\rangle}=-z_{\langle dr\rangle}\), we find
\[
(\widehat Z_N)_{du,dv}
=
d\sum_{r=1}^{N_d^-}z_{\langle dr \rangle}\sin\!\Bigl(\frac{2\pi uv^{-1}r}{N_d}\Bigr)
=
d\,\widehat z_d(uv^{-1}).
\]
\end{proof}

\begin{remark}
Lemma~\ref{lemma:structure-hats} may be read as a support statement for the
DFT. For a fixed column index \(k\), the corresponding
column of \(X_N\) factors through the quotient
\(\mathbb Z/N\mathbb Z\twoheadrightarrow \mathbb Z/(N/\gcd(k,N))\mathbb Z\), so
its transform is supported only on frequencies divisible by \(\gcd(k,N)\). The
formulas for \(Y_N\) and \(Z_N\) are the even and odd counterparts of the same
observation.
\end{remark}

\subsection{Block structure of the transformed matrices}
The preceding lemma immediately imposes a block structure on the transformed
matrices after suitable permutation. The diagonal blocks are indexed by divisors $d$ of $N$, and on each such
block one recovers a group determinant.

\begin{lemma}\label{lemma:block-structure}
For each of \(\widehat X_N\), \(\widehat Y_N\), and \(\widehat Z_N\), group rows
and columns according to the value of \(\gcd(\,\cdot\,,N)\), and order the
divisors of \(N\) by reverse divisibility. Then the resulting matrix is block
upper triangular.

Fix $d\mid N$, write $G_{N_d}:=(\mathbb Z/N_d\mathbb Z)^\times$, and choose a set
$T_{N_d}\subset G_{N_d}$ of representatives for $G_{N_d}/\{\pm1\}$.
Define
\begin{align*}
G_d^{(X)}
&=\bigl(d\,\widehat x_d(uv^{-1})\bigr)_{u,v\in G_{N_d}},\\
G_d^{(Y)}
&=\bigl(d\,\widehat y_d(uv^{-1})\bigr)_{u,v\in T_{N_d}},\\
G_d^{(Z)}
&=\bigl(d\,\widehat z_d(uv^{-1})\bigr)_{u,v\in T_{N_d}}.
\end{align*}
Then we have
\[
\det(\widehat X_N)=\prod_{d\mid N}\det G_d^{(X)},\qquad
\det(\widehat Y_N)=\prod_{d\mid N}\det G_d^{(Y)},\qquad
\det(\widehat Z_N)=\prod_{\substack{d\mid N\\ N_d>2}}\det G_d^{(Z)}.
\]
\end{lemma}

\begin{proof}
If an entry in row \(m\) and column \(k\) is nonzero, then
Lemma~\ref{lemma:structure-hats} shows that \(\gcd(k,N)\mid m\), hence
\(\gcd(k,N)\mid \gcd(m,N)\). 
This means any nonzero entry lies in a block with row block index $\le$ column block index in the chosen order, 
which is precisely the definition of block upper triangularity.

It follows that the determinant of each matrix is the product of the
determinants of its diagonal blocks. We now describe the contribution of the
diagonal block corresponding to a fixed divisor $d$.

For $\widehat X_N$, the row and column indices in the diagonal block
corresponding to $d$ are of the form $du$ and $dv$ with $u, v \in G_{N_d}$. By
Lemma~\ref{lemma:structure-hats}, the corresponding entries are
$d\,\widehat x_d(uv^{-1})$. Hence the determinant of this diagonal block is
$\det G_d^{(X)}$.

For $\widehat Y_N$, the entries of the diagonal block corresponding to $d$ are of the form
$d\,\widehat y_d(uv^{-1})$. Since $\widehat y_d(-t)=\widehat
y_d(t)$, after choosing representatives $T_{N_d}$ for $G_{N_d}/\{\pm1\}$, this
gives the matrix $G_d^{(Y)}$, up to simultaneous permutations of rows and
columns. Hence the determinant of the diagonal block corresponding to $d$ is
$\det G_d^{(Y)}$.

For $\widehat Z_N$, the block corresponding to \(d\) is absent when \(N_d=1\) or \(2\), assume therefore that $N_d>2$. 
By Lemma~\ref{lemma:structure-hats}, the entries of the diagonal block corresponding to $d$ are of the form
$d\,\widehat z_d(uv^{-1})$. After choosing representatives $T_{N_d}$ for
$G_{N_d}/\{\pm1\}$, this gives the matrix $G_d^{(Z)}$, up to simultaneous
permutations of rows and columns and simultaneous sign changes of rows and columns coming
from $\widehat z_d(-t)=-\widehat z_d(t)$. Hence the determinant of the diagonal
block corresponding to $d$ is $\det G_d^{(Z)}$. If $N_d=1$ or $2$, the
corresponding block is absent.

Multiplying these contributions gives the stated formulas.
\end{proof}

We now derive a formula for $\det(Y_N'|_{y_0=0})$.
\begin{lemma}\label{lemma:Y-truncation}
Let the notation be as in Lemma \ref{lemma:block-structure}. Then we have
\[
\det(C_N)\det\!\bigl(Y_N'|_{y_0=0}\bigr)
=
\frac{N}{2}\prod_{\substack{d\mid N\\ d\neq N}} \det(G_d^{(Y)}|_{y_0=0}).
\]
\end{lemma}
\begin{proof}
Write
\[
Y_N=
\begin{pmatrix}
y_0 & y_0\mathbf 1^{T}\\
y_0\mathbf 1 & Y_N'
\end{pmatrix},
\]
where \(\mathbf 1\) is the column vector of length \(N^+\) all of whose entries are \(1\).

Subtracting the first row from each of the remaining rows, we get
\[
\det(Y_N)
=
\det\begin{pmatrix}
y_0 & y_0\mathbf 1^{T}\\
0 & Y_N'-y_0J
\end{pmatrix}
=
y_0\det(Y_N'-y_0J),
\]
where \(J\) is the all-ones matrix of size \(N^+\times N^+\).
Hence
\[
\left.\frac{\det(Y_N)}{y_0}\right|_{y_0=0}
=
\det\!\bigl(Y_N'|_{y_0=0}\bigr).
\]

By Lemma~\ref{lemma:block-structure}, we have 
\[
\det(C_N)\det(Y_N)=\prod_{d\mid N}\det G_d^{(Y)}.
\]

The block corresponding to \(d=N\) is \(1\times 1\) with the entry $N\widehat y_N(0)=\frac N2\,y_0$.
Thus
\[
\prod_{d\mid N}\det G_d^{(Y)}
=
\frac N2\,y_0
\prod_{\substack{d\mid N\\ d\neq N}}\det G_d^{(Y)}.
\]

Dividing by \(y_0\) and then setting \(y_0=0\), we obtain
\[
\det(C_N)\det\!\bigl(Y_N'|_{y_0=0}\bigr)
=
\frac{N}{2}\prod_{\substack{d\mid N\\ d\neq N}}
\det\!\bigl(G_d^{(Y)}|_{y_0=0}\bigr).
\]
\end{proof}

We next record the determinants of the three transform matrices.
\begin{lemma}
We have
\begin{align}
\det(F_N)
&=
\mathrm{i}^{\frac{(N-1)(3N-2)}{2}}\,N^{\frac{N}{2}}, \label{eq:detF} \\
\det(C_N)
&=
(-1)^{\frac{N^+(N^++1)}{2}}\,\frac{N^{\frac{N^++1}{2}}}{2^{N^++1}}, \label{eq:detC} \\
\det(S_N)
&=
(-1)^{\frac{N^-(N^--1)}{2}}\,\frac{N^{\frac{N^-}{2}}}{2^{N^-}} \label{eq:detS}.
\end{align}
\end{lemma}

\begin{proof}
The formula for $\det(F_N)$ is the classical Vandermonde evaluation of the
Fourier matrix. 

For $\det(S_N)$, we can take $X_i=w_N^i$ in \cite[Lemma~2, (2.3)]{K99}. After some simplifications, we get
\begin{equation*}
\det(S_N)
=2^{N^-(N^- -1)}
\prod_{1\le i<j\le N^-}
\sin\!\left(\frac{\pi(i-j)}{N}\right)
\sin\!\left(\frac{\pi(i+j)}{N}\right)
\prod_{i=1}^{N^-}\sin\!\left(\frac{2\pi i}{N}\right).
\end{equation*}
From this formula, we see the sign of $\det(S_N)$ is $(-1)^{\frac{N^-(N^--1)}{2}}$.

On the other hand, it is well known $S_N^2=\frac{N}{4}I_{N^-}$ which shows 
$\det(S_N)^2 = \left(N/4\right)^{N^-}$. Combining these two facts, we get the formula for $\det(S_N)$.

The formula for $\det(C_N)$ can be obtained similarly using \cite[Lemma~2, (2.5)]{K99} and $C_N^2=\frac{N}{4}I_{N^++1}$.
\end{proof}

We can now combine the block decomposition with the standard factorization of a
group determinant into linear factors indexed by characters.

\begin{theorem}\label{thm:det-XYZ-general}
For each divisor \(d\mid N\), let $\widehat G_{N_d}$ be the character group of $G_{N_d}$ and 
\[
\widehat G_{N_d}^{+}:=\{\chi\in\widehat G_{N_d}:\chi(-1)=1\},\qquad
\widehat G_{N_d}^{-}:=\{\chi\in\widehat G_{N_d}:\chi(-1)=-1\}.
\]
Then
\begin{align*}
\det(X_N)
&=
\mathrm{i}^{-\frac{(N-1)(3N-2)}{2}}\,N^{-\frac{N}{2}}
\prod_{d\mid N}\ \prod_{\chi\in\widehat G_{N_d}}
\left(
d\sum_{t\in G_{N_d}}\widehat x_d(t)\chi(t)
\right),\\
\det(Y_N)
&=
(-1)^{\frac{N^+(N^++1)}{2}}\,
\frac{2^{N^+ +1 - N^-}}{N^{\frac{N^++1}{2}}}
\prod_{d\mid N}\ \prod_{\chi\in\widehat G_{N_d}^{+}}
\left(
d\sum_{t\in G_{N_d}}\widehat y_d(t)\chi(t)
\right),\\
\det\left(Y_N'|_{y_0=0}\right)
&=
(-1)^{\frac{N^+(N^++1)}{2}}
\frac{2^{N^+-N^-}}{N^{\frac{N^+-1}{2}}}
\prod_{\substack{d\mid N\\ d\neq N}}
\ \prod_{\chi\in \widehat G_{N_d}^{+}}
\left(
d\sum_{t\in G_{N_d}}
\bigl(\widehat y_d(t)|_{y_0=0}\bigr)\chi(t)
\right) \\
\det(Z_N)
&=
(-1)^{\frac{N^-(N^--1)}{2}}\,
N^{-\frac{N^-}{2}}
\prod_{\substack{d\mid N\\ N_d>2}}\ \prod_{\chi\in\widehat G_{N_d}^{-}}
\left(
d\sum_{t\in G_{N_d}}\widehat z_d(t)\chi(t)
\right).
\end{align*}
\end{theorem}

\begin{proof}
By Lemma~\ref{lemma:block-structure},
\[
\det(\widehat X_N)=\prod_{d\mid N}\det G_d^{(X)},\qquad
\det(\widehat Y_N)=\prod_{d\mid N}\det G_d^{(Y)},\qquad
\det(\widehat Z_N)=\prod_{\substack{d\mid N\\ N_d>2}}\det G_d^{(Z)}.
\]

For \(X_N\), the block corresponding to \(d\) is
\[
G_d^{(X)}=\bigl(d\,\widehat x_d(uv^{-1})\bigr)_{u,v\in G_{N_d}},
\]
which is the usual group matrix of \(G_{N_d}\) attached to the function
\(t\mapsto d\,\widehat x_d(t)\). By the Dedekind factorization of a group
determinant,
\[
\det G_d^{(X)}
=
\prod_{\chi\in\widehat G_{N_d}}
\left(
d\sum_{t\in G_{N_d}}\widehat x_d(t)\chi(t)
\right).
\]
Multiplying over all \(d\mid N\), and using
\[
\det(\widehat X_N)=\det(F_N)\det(X_N),
\]
together with \eqref{eq:detF}, yields the formula for $\det(X_N)$.

For \(Y_N\), set
$H_{N_d}:=G_{N_d}/\{\pm1\}.$
Since \(\widehat y_d(-t)=\widehat y_d(t)\), the function
\(t\mapsto d\,\widehat y_d(t)\) descends to a well-defined function on \(H_{N_d}\).
The block
\[
G_d^{(Y)}=\bigl(d\,\widehat y_d(uv^{-1})\bigr)_{u,v\in T_{N_d}}
\]
is precisely the group matrix of \(H_{N_d}\) attached to this descended function.
The character group of \(H_{N_d}\) is naturally identified with
\(\widehat G_{N_d}^{+}\). Hence
\[
\det G_d^{(Y)}
=
\prod_{\chi\in\widehat G_{N_d}^{+}}
\left(
d\sum_{t\in H_{N_d}}\widehat y_d(t)\chi(t)
\right).
\]

Multiplying over all \(d\mid N\), using the fact $\sum_{t \in G_{N_d}} \widehat{y}_d(t) \chi(t)=2 \sum_{t \in H_{N_d}} \widehat{y}_d(t) \chi(t)$ for $N_d>2$ and
\[
\det(\widehat Y_N)=\det(C_N)\det(Y_N),
\]
together with \eqref{eq:detC}, gives the formula for $\det(Y_N)$.

The formula for \(\det(Y_N'|_{y_0=0})\) follows similarly from Lemma~\ref{lemma:Y-truncation}, the factorization
of $\det(G_d^{(Y)})$ and \eqref{eq:detC}.

For \(Z_N\), we have \(N_d>2\). 
For each odd character \(\chi\in\widehat G_{N_d}^{-}\), define
$$
w_\chi:=\left(\chi(x)^{-1}\right)_{x \in T_{N_d}} .
$$


Since \(\widehat z_d(-t)=-\widehat z_d(t)\) and \(\chi(-t)=-\chi(t)\), the product
\(\widehat z_d(t)\chi(t)\) depends only on the class \([t]\in H_{N_d}\). 
Therefore
\begin{align*}
\left(G_d^{(Z)} w_\chi\right)_u
&=
\sum_{v \in T_{N_d}} d \widehat{z}_d\left(u v^{-1}\right) \chi(v)^{-1} \\
&=
\chi(u)^{-1}\sum_{v \in T_{N_d}} d \widehat{z}_d\left(u v^{-1}\right)\chi\left(u v^{-1}\right) \\
&=
\left(d \sum_{t \in H_{N_d}} \widehat{z}_d(t) \chi(t)\right)\chi(u)^{-1}.
\end{align*}

Hence $w_\chi$ is an eigenvector of $G_d^{(Z)}$ with eigenvalue
$$
d \sum_{t \in H_{N_d}} \widehat{z}_d(t) \chi(t) .
$$

Since the odd characters form a set of cardinality \(\varphi(N_d)/2=|T_{N_d}|\), and they are
linearly independent, the vectors \(w_\chi\) form a basis of \(\mathbb C^{T_{N_d}}\).
It follows that
\[
\det G_d^{(Z)}
=
\prod_{\chi\in\widehat G_{N_d}^{-}}
\left(
d\sum_{t\in H_{N_d}}\widehat z_d(t)\chi(t)
\right).
\]
Multiplying over \(d\mid N\) with $N_d>2$, using the fact 
$$\sum_{t \in G_{N_d}} \widehat{z}_d(t) \chi(t)=2 \sum_{t \in H_{N_d}} \widehat{z}_d(t) \chi(t)$$ 
and
\[
\det(\widehat Z_N)=\det(S_N)\det(Z_N),
\]
together with \eqref{eq:detS}, gives the formula for $\det(Z_N)$.
\end{proof}
\begin{remark}\label{rk:dpair}
In Theorem~\ref{thm:det-XYZ-general}, the factor indexed by $d \mid N$
is naturally attached to the quotient modulus $N_d = N/d$.
For this reason, the character sums and group determinants arising there
are first expressed in terms of $N_d$.
In the later explicit formulas, such as Theorems~\ref{thm:det_B_formulas}
and~\ref{thm:det_T_formulas}, it is more convenient to write the final products
in terms of characters modulo $d$.
This causes no change, since replacing $N_d$ by $d$ in a product over divisors $d \mid N$ merely permutes the indexing set.
\end{remark}

\section{Applications of Theorem \ref{thm:det-XYZ-general}}
\subsection{Determinant of Bernoulli matrices}
We first recall the formula for the discrete sine and cosine transforms of periodic Bernoulli polynomials.
\begin{lemma}\label{lemma:Fourier-Bernoulli}
Let $y_{[a]}=\widetilde{B}_k(a / N)$ for $k$ even and $z_{\langle a\rangle}=\widetilde{B}_k(a / N)$ for $k$ odd. Let $d \mid N$, $N_d=N/d$ and $1\le t\le N_d-1$.
Then for $k$ even, we have
\begin{equation*}
\widehat y_N(0) = \frac{1}{2}B_k(0) = \frac{(-1)^{\frac{k}{2}+1}k!}{(2\pi)^{k}}\zeta(k),
\end{equation*}
and 
\begin{equation*}
\widehat y_d(t)
= \frac{(-1)^{\frac{k}{2}+1}k!}{2(2\pi)^{k}}\,N_d^{1-k}
\left(
\zeta\!\left(k,\frac{t}{N_d}\right)
+
\zeta\!\left(k,1-\frac{t}{N_d}\right)
\right),
\end{equation*}
where $\zeta(z)$ is the Riemann zeta function and $\zeta(z,a)$ is the Hurwitz zeta function.

For odd $k \neq 1$, we have
\begin{equation*}
\widehat z_d(t) =\frac{(-1)^{\frac{k+1}{2}} k!}{2(2\pi)^{k}}\,N_d^{1-k}
\left(
\zeta\!\left(k,\frac{t}{N_d}\right)
-
\zeta\!\left(k,1-\frac{t}{N_d}\right)
\right).
\end{equation*}
For $k = 1$, we have
\begin{align*}
\widehat z_d(t)=-\frac14\cot\!\Bigl(\frac{\pi t}{N_d}\Bigr).
\end{align*}
\end{lemma}
\begin{proof}
By definition, we have 
\begin{align*}
\widehat y_d(t) = \frac{1}{2}\sum_{r=1}^{N_d}\widetilde{B}_k(r/N_d)\cos\!\Bigl(\frac{2\pi tr}{N_d}\Bigr),\\ 
\widehat z_d(t) = \frac{1}{2}\sum_{r=1}^{N_d}\widetilde{B}_k(r/N_d)\sin\!\Bigl(\frac{2\pi tr}{N_d}\Bigr).
\end{align*}
The lemma is a direct application of parts (iv) and (iii) of \cite[Corollary 2]{CK00}, in the cases where 
$k$ is even and $t\neq 0$, and where $k\neq 1$ is odd, respectively.
The formula for $\widehat y_N(0)$ is classical.
For $k=1$, we apply \cite[page 14]{RG72}.
\end{proof}

\begin{theorem}\label{thm:det_B_formulas}
For $k$ even, we have
\begin{equation*}
\det(B_{k,N})
=
(-1)^{\frac{N^+(N^++1)}{2}} N^{\frac{N^+ + 1}{2}} 2^{N^+ +1 - N^-}\cdot
\prod_{d\mid N}\ \prod_{\substack{\chi\ (\mathrm{mod}\ d)\\ \chi(-1)=1}}
\left(\frac{(-1)^{\frac{k}{2}+1}k!}{(2\pi)^k}\,L(k,\chi)\right),
\end{equation*}
where for $d = 1$, the unique character is trivial and $L(k,\chi)=\zeta(k)$.

For $k$ odd, we have
\begin{equation*}
\det(B_{k,N})
=
(-1)^{\frac{N^-(N^- -1)}{2}} N^{\frac{N^-}{2}}\cdot
\prod_{\substack{d\mid N \\ d>2}}\ \prod_{\substack{\chi\ (\mathrm{mod}\ d)\\ \chi(-1)=-1}}
\left(\frac{(-1)^{\frac{k+1}{2}}k!}{(2\pi)^k}\,L(k,\chi)\right).
\end{equation*}
In particular, $\det(B_{k,N})$ does not vanish.
\end{theorem}
\begin{proof}
Let $\chi$ be a character modulo $N_d$. We have the following fact by expressing 
$L(k,\chi)$ as $\sum_{n\ge 1} \chi(n)n^{-k}$ and grouping terms by residue classes modulo $N_d$
\begin{equation*}
\sum_{t=1}^{N_d-1} \chi(t)(\zeta(k, t / N_d) \pm \zeta(k, 1-t / N_d))=2 N_d^k L(k, \chi)
\end{equation*}
where the sign is plus when $k$ and $\chi$ are even, and minus when $k\neq 1$ and $\chi$ are odd. For $k=1$ and $\chi$ odd, we have
$\sum_{t=1}^{N_d-1} \chi(t) \cot (\pi t / N_d) = \frac{2}{\pi}N_d L(1, \chi)$ (see \cite[Corollary 6]{F90}).

Then the formula for $\det(B_{k,N})$ is a direct application of Theorem~\ref{thm:det-XYZ-general}, Lemma~\ref{lemma:Fourier-Bernoulli} 
and the reindexing explained in Remark~\ref{rk:dpair}.

Since $L(k,\chi)\neq 0$ for $k \ge 1$, the determinant does not vanish.
\end{proof}
\begin{remark}
  Note that the products in the above theorem are Galois invariant, so the results are rational as expected.
\end{remark}

\subsection{Determinant of tangent matrices and a conjecture of Sun}
We first give a lemma on the discrete sine and cosine transforms of the powers of the tangent function.
\begin{lemma}\label{lemma: Fourier-tangent}
Assume $N=2n+1$. For $0\le a\le n$ and $k\ge 1$, set
$$y_{[a]}=\tan^{2k}\!\left(\frac{\pi a}{N}\right),$$
and for $1\le a\le n$ and $k\ge 0$, set
$$z_{\langle a\rangle}=\tan^{2k+1}\!\left(\frac{\pi a}{N}\right).$$
For \(d\mid N\),  $d\neq N$, put \(N_d=N/d\) and \(1\le t\le N_d-1\).

Then we have 
\begin{align*}
\widehat y_d(t)
&= \sum_{s=1}^{k} a_{k,s}\, N_d^{2s} \left( \widetilde{B}_{2s}\!\left(\frac{t}{N_d}\right) - 2^{2s} \widetilde{B}_{2s}\!\left(\frac{(n+1) t}{N_d}\right) \right), \\
\widehat z_d(t) &= \sum_{s=0}^{k} A_{k,s}\, N_d^{2s+1} \left( \widetilde{B}_{2s+1}\!\left(\frac{t}{N_d}\right) - 2^{2s+1} \widetilde{B}_{2s+1}\!\left(\frac{(n+1) t}{N_d}\right) \right).
\end{align*}
The rational coefficients $a_{k,s}$ and $A_{k,s}$ are universal constants independent of $N$ and $d$, given by the following formulas
\begin{align*}
a_{k,s} &= \frac{(-1)^{k}}{2(2k-1)!}\binom{2k-1}{2s-1}\frac{1}{2s} \sum_{\alpha=0}^{2k}\binom{2k}{\alpha}\,B^{(2k)}_{\,2k-2s}(\alpha) \\
A_{k,s} &= \frac{(-1)^{k+1}}{2(2k)!}\binom{2k}{2s}\frac{1}{2s+1} \sum_{\beta=0}^{2k+1}\binom{2k+1}{\beta}\,B^{(2k+1)}_{\,2k-2s}(\beta)
\end{align*}
where $B_r^{(q)}(x)$ is the higher-order Bernoulli polynomial defined by
$$\left(\frac{u}{e^u-1}\right)^q e^{xu} = \sum_{r=0}^{\infty} B_r^{(q)}(x)\frac{u^r}{r!}.$$
\end{lemma}
\begin{proof}
The formulas for these discrete sine and cosine transforms are given in \cite{C08} (see also \cite[(2.7a), (2.8a)]{CS12}) using Euler polynomials and higher-order Bernoulli polynomials. We can replace the Euler polynomials in those formulas by Bernoulli polynomials using the well-known relation
\begin{equation*}
E_{m-1}(x)=\frac{2}{m}\left(B_{m}(x)-2^{m} B_{m}(x / 2)\right).
\end{equation*}
We claim 
\begin{equation}\label{eq:B}
B_m\!\left(\frac{t}{N_d}\right)-2^m B_m\!\left(\frac{t}{2N_d}\right) = (-1)^t\left(\widetilde B_m\!\left(\frac{t}{N_d}\right)-2^m \widetilde B_m\!\left(\frac{(n+1)t}{N_d}\right)\right).
\end{equation}
Combining these two formulas, we obtain our result.

Now we prove the claim. Since
$$\frac{(n+1)t}{N_d}=\frac{dt}{2}+\frac{t}{2N_d},$$
we distinguish two cases according to the parity of $t$.

If $t$ is even, then \(t/2\in\mathbb Z\), so by periodicity of $\widetilde B_m(x)$ we prove the claim.

If \(t\) is odd, then
\[
\frac{(n+1)t}{N_d}\equiv \frac12+\frac{t}{2N_d}\pmod{1}.
\]
Applying the duplication formula for the Bernoulli polynomial
\[
B_m(x)=2^{m-1}\!\left(B_m(x/2)+B_m\!\left(\frac{x+1}{2}\right)\right)
\]
with \(x=t/N_d\), we get
\begin{equation*}
B_m\!\left(\frac{t}{N_d}\right)-2^mB_m\!\left(\frac12+\frac{t}{2N_d}\right)
=
-\left(B_m\!\left(\frac{t}{N_d}\right)-2^mB_m\!\left(\frac{t}{2N_d}\right)\right).
\end{equation*}
\end{proof}

\begin{example}
For small exponents, it is useful to have the coefficients $a_{k,s}$ and $A_{k,s}$ written out explicitly. 
For $k \le 4$, the values of $a_{k,s}$ are:
    \begin{align*}
        a_{1,1} &= \textstyle -1, \\
        a_{2,1} &= \frac{4}{3}, a_{2,2}=\frac{1}{3}, \\
        a_{3,1} &= -\frac{23}{15}, a_{3,2}= -\frac{2}{3}, a_{3,3}= -\frac{2}{45}, \\
        a_{4,1} &= \frac{176}{105}, a_{4,2}=\frac{44}{45}, a_{4,3}=\frac{16}{135}, a_{4,4}=\frac{1}{315}, \\
    \end{align*}
 and the values of $A_{k,s}$ are:
    \begin{align*}
        A_{0,0} &= -1, \\
        A_{1,0} &= 1, A_{1,1} = \frac{2}{3}, \\
        A_{2,0} &= -1, A_{2,1}=-\frac{10}{9}, A_{2,2}=-\frac{2}{15}, \\
        A_{3,0} &= 1, A_{3,1}=\frac{196}{135}, A_{3,2}=\frac{14}{45}, A_{3,3}=\frac{4}{315}, \\
        A_{4,0} &= -1, A_{4,1}=-\frac{1636}{945}, A_{4,2}=-\frac{38}{75}, A_{4,3}=-\frac{4}{105}, A_{4,4}=-\frac{2}{2835}.
    \end{align*}
\end{example}

\begin{example}\label{ex:yzhat}
Applying Lemma~\ref{lemma: Fourier-tangent} and \eqref{eq:B}, for $k=1$ in the even case and $k=0$ in the odd case, we have, for \(1\le t\le N_d-1\), 
$$d\widehat y_d(t) = (-1)^{t+1}N N_d\left(B_2\left(\frac{t}{N_d}\right)-4 B_2\left(\frac{t}{2N_d}\right)\right)= \frac{N}{2}(-1)^{t+1}(2t-N_d),$$
and
$$d \widehat z_d(t) = (-1)^{t+1}N\left(B_1\left(\frac{t}{N_d}\right)-2 B_1\left(\frac{t}{2N_d}\right)\right)=\frac{N}{2}(-1)^{t+1}.$$
\end{example}

Now we can prove a corrected version of a conjecture of Sun \cite[Conjecture 5.2]{S24}. 
The first statement in the following corollary strengthens his conjecture, while the second corrects the power of 4 in that conjecture.
\begin{corollary}\label{cor:Sun}
With the above notation, we have the following:
\begin{enumerate}
\item[(i)]
Let $s_n=N^{-\frac{n}{2}} \det(T_{1,N})$. Then $s_n \in 2^{n+1-\tau(N)} \Z$ where $\tau(N)$ is the number of divisors of $N$.
\item[(ii)] We have
$$
\det(T_{2,N}) \in N^{\frac{n+1}{2}} 4^{n+1-\tau(N)} \Z .
$$
\item[(iii)]  There is a matrix $T(n)=\left(t_{j k}\right)_{1 \leqslant j, k \leqslant n}$ with entries in $\{0, \pm 1\}$ such that
$$
2 \sum_{k=1}^n t_{j k} \sin \frac{\pi k}{N}=\tan \frac{\pi j}{N} \quad \text { for all } j=1, \ldots, n .
$$
Let $t_n=\det(T(n))$. Then $s_n=-t_n$ if $n \equiv 3\pmod 4$, and $s_n=t_n$ otherwise.
\end{enumerate}
\end{corollary}
\begin{proof}
(i) By Lemma~\ref{lemma:structure-hats}, the matrix $\widetilde{T}(1,N)=\frac{2}{N}S_N T_{1,N}$ has entries in $\{0, \pm 1\}$. 
By Lemma~\ref{lemma:block-structure}, its determinant is the product of group determinants of size $\varphi(N_d)/2$ indexed by $d \mid N, d\neq N$ with entries $\pm 1$.

In each block of size $\varphi(N_d) / 2$, subtract the first row from every other row; each new non-first row then has entries in $\{0, \pm 2\}$, so one factors out 2 from each of the remaining $\varphi(N_d) / 2-1$ rows. Moreover, 
$$ \sum_{\substack{d \mid N \\ d \neq N}}\left(\frac{\varphi\left(N_d\right)}{2}-1\right)=\sum_{\substack{e \mid N \\ e>1}}\left(\frac{\varphi(e)}{2}-1\right)=n+1-\tau(N).$$

Taking the product, we see 
$$\left(\frac{2}{N}\right)^n \det(S_N) \det(T_{1,N}) \in 2^{n+1-\tau(N)}\Z.$$
Plugging in the formula for $\det(S_N)$ in \eqref{eq:detS} proves the result.

(ii) By Lemma~\ref{lemma:Y-truncation} and Example~\ref{ex:yzhat}, we see that
\[
\left(\frac{2}{N}\right)^{n+1}\det(C_N)\det(T_{2,N})
\]
is the product of the group determinants indexed by \(d\mid N\), \(d\neq N\), with entries \( (-1)^{t+1}(2t-N_d)\), where \(t=uv^{-1}\) and \(u,v\in (\mathbb Z/N_d\mathbb Z)^\times/\{\pm1\}\).

Note that $(-1)^{t+1}(2t-N_d)\equiv N_d \pmod 4$ since $N_d$ is odd. As in the proof of (i), we can show
$$\left(\frac{2}{N}\right)^{n+1}\det(C_N)\det(T_{2,N}) \in 4^{n+1-\tau(N)}\Z.$$

Plugging in the formula for $\det(C_N)$ in \eqref{eq:detC} gives the desired result.

(iii) Taking the transpose of $\widetilde{T}(1,N)$, we have $\widetilde{T}(1,N)^{T}=\frac{2}{N} T_{1,N}S_N$. Multiplying both sides by $S_N$ and using the fact $S_N^2=\frac{N}{4}I_n$, we get
\begin{equation}\label{eq:TST}
2 \widetilde{T}(1,N)^{T} S_N = T_{1,N}.
\end{equation}
The first column of $S_N$ is $(\sin \frac{2 \pi k}{N})_{1 \le k \le n}^T$ which is a permutation of $(\sin \frac{\pi k}{N})_{1 \le k \le n}^T$.
Specifically, the permutation is 
\begin{equation*}
\begin{pmatrix}
1 & 2 & \cdots & n/2 & n/2+1 &  n/2+2 & \cdots & n \\
2 & 4 & \cdots & n   & n-1 &  n-3 & \cdots & 1 
\end{pmatrix} 
\quad \text{for } n \text{ even,}
\end{equation*}
and 
\begin{equation*}
\begin{pmatrix}
1 & 2 & \cdots & (n-1)/2 & (n+1)/2 & (n+3)/2 & \cdots & n \\
2 & 4 & \cdots & n-1   &  n & n-2 & \cdots & 1 
\end{pmatrix}
\quad \text{for } n \text{ odd.}
\end{equation*}
It is straightforward to check that the permutation is odd if and only if $n \equiv 2 \pmod 4$. Let $P$ be the matrix representing this permutation. 
Then the first column of $P^{-1} S_N$ is exactly $(\sin \frac{\pi k}{N})_{1 \le k \le n}^T$.
Hence we can take $T(n)$ to be $\widetilde{T}(1, N)^T P$, since comparing the first columns in \eqref{eq:TST} gives the required identity.
We have
\begin{align*}
  t_n &= \det(\widetilde{T}(1,N)) \det(P) \\
   & = (-1)^{\frac{n(n-1)}{2}}\det(P) N^{-\frac{n}{2}}\det(T_{1,N}) \\
   & = (-1)^{\frac{n(n-1)}{2}}\det(P) s_n
\end{align*}
Then the result follows from the fact that $(-1)^{\frac{n(n-1)}{2}}\det(P)$ equals $-1$ if $n \equiv 3\pmod 4$ and equals $1$ otherwise.
\end{proof}

Let $\chi$ be a Dirichlet character modulo $N_{\chi}$. Then the generalized Bernoulli number of $\chi$ is defined by
\begin{equation*}
B_{m,\chi} = N_\chi^{m-1} \sum_{a=1}^{N_{\chi}} \chi(a) B_m\left(\frac{a}{N_{\chi}}\right).
\end{equation*}
Let $\chi^*$ be the primitive Dirichlet character of conductor $f_\chi$ that induces $\chi$. 
It is well known that
\begin{equation}\label{eq:primtoimprim}
B_{m,\chi} = B_{m,\chi^*} \prod_{p \mid N_{\chi}} \Bigl(1 - \chi^*(p)p^{m-1}\Bigr).
\end{equation}


\begin{theorem} \label{thm:det_T_formulas}
Let the notation be as above. We have the following factorizations of determinants
\begin{equation*} 
\det(T_{2k,N}) = (-1)^{\frac{n(n+1)}{2}} N^{\frac{n+1}{2}} \prod_{\substack{d \mid N \\ d > 2}} \;\; \prod_{\substack{\chi \bmod d \\ \chi(-1)=1}} \widetilde{\lambda}_{\chi}^{(2k,d)},
\end{equation*}
where
\begin{equation*}
\widetilde{\lambda}_{\chi}^{(2k,d)} = \sum_{s=1}^{k} a_{k,s} \Bigl(1 - 2^{2s}\chi(2)\Bigr) B_{2s,\chi^*} \prod_{p \mid d} \Bigl(1 - \chi^*(p)p^{2s-1}\Bigr),
\end{equation*}
and
\begin{equation*} 
\det(T_{2k+1,N}) = (-1)^{\frac{n(n-1)}{2}} N^{\frac{n}{2}} \prod_{\substack{d \mid N \\ d > 2}} \;\; \prod_{\substack{\chi \bmod d \\ \chi(-1)=-1}} \widetilde{\lambda}_{\chi}^{(2k+1, d)},
\end{equation*}
where
\begin{equation*}
\widetilde{\lambda}_{\chi}^{(2k+1, d)} = \sum_{s=0}^{k} A_{k,s} \Bigl(1 - 2^{2s+1}\chi(2)\Bigr) B_{2s+1,\chi^*} \prod_{p \mid d} \Bigl(1 - \chi^*(p)p^{2s}\Bigr).
\end{equation*}
Here, $a_{k,s}$ and $A_{k,s}$ are universal rational coefficients as before, and the product over $p$ runs over all prime divisors of $d$.

As a consequence,  we have
\begin{equation*} 
\det(T_{1,N}) = (-1)^{\frac{n(n+1)}{2}} 2^n N^{\frac{n}{2}} \prod_{\substack{d \mid N \\ d > 2}} \left[ \frac{h^{-}(\Q(\zeta_d))}{2d Q_{\Q(\zeta_d)}} \prod_{\substack{\chi \bmod d \\ \chi(-1)=-1}} \left(\Bigl(2\chi(2)-1\Bigr) \prod_{p \mid d} \Bigl(1 - \chi^*(p)\Bigr)\right)\right],
\end{equation*}
where $h^{-}(\Q(\zeta_d))$ is the relative class number of $\Q(\zeta_d)$, and $Q_{\Q(\zeta_d)}$ is the Hasse unit index which equals 1 if $d$ is a prime power and 2 otherwise.
\end{theorem}
\begin{proof} 
Let $\chi$ be a character modulo $N_d$. Since $(n+1) \equiv 2^{-1} \pmod {N_d}$, by Lemma \ref{lemma: Fourier-tangent} and \eqref{eq:primtoimprim}, we have
\begin{align*}
d\sum_{t=1}^{N_d} \chi(t) \widehat z_d(t) &= d\sum_{t=1}^{N_d}\chi(t)\sum_{s=0}^{k} A_{k,s}\, N_d^{2s+1} \left(\widetilde{B}_{2s+1}\!\left(\frac{t}{N_d}\right) - 2^{2s+1} \widetilde{B}_{2s+1}\!\left(\frac{(n+1) t}{N_d}\right)\right) \\
&= d \sum_{s=0}^{k} A_{k,s}\, N_d^{2s+1} \Bigl(1-2^{2s+1}\chi(2)\Bigr) N_d^{-2s} B_{2s+1,\chi} \\
&= N \sum_{s=0}^{k} A_{k,s}\, \Bigl(1-2^{2s+1}\chi(2)\Bigr) B_{2s+1,\chi^*} \prod_{p \mid N_d} \Bigl(1 - \chi^*(p)p^{2s}\Bigr).
\end{align*}
Then the formula for $\det(T_{2k+1,N})$ is a direct application of Theorem \ref{thm:det-XYZ-general} and the reindexing explained in Remark~\ref{rk:dpair}.

Similarly, we compute $d\sum_{t=1}^{N_d} \chi(t) \widehat y_d(t)$, and then derive the formula for $\det(T_{2k,N})$ by the same argument.

We obtain the formula for $\det(T_{1,N})$ using the class number formula (see \cite[Proposition 4.9]{W82})
\begin{equation*}
  h^{-}(\Q(\zeta_d)) = \frac{2d Q_{\Q(\zeta_d)}}{(-2)^{\varphi(d)/2}}\prod_{\substack{\chi \bmod d \\ \chi(-1)=-1}}B_{1, \chi^*}.
\end{equation*}
\end{proof}
\begin{remark}
The same strategy also works for other trigonometric determinants using \cite{CS12}. 
For example, we can also obtain factorizations of the determinants of the following matrices
$$\left(\sec^{2k}\left(\frac{\pi i j}{N}\right)\right)_{1\le i,j \le n}, \left((-1)^{ij}\tan^{m}\left(\frac{\pi i j}{N}\right)\right)_{1\le i,j \le n}, \left((-1)^{ij}\sec^{2k}\left(\frac{\pi i j}{N}\right)\right)_{1\le i,j \le n}.$$
\end{remark}

\subsection*{Acknowledgements}
I thank the anonymous referee of \cite{BDLR24} for raising the question of the non-vanishing of $\det(B_{3,N})$, which motivated the present note.
I also thank Fran{\c{c}}ois Brunault, Rob de Jeu and Sun Zhi-Wei for very helpful discussions and comments.

This research was supported by the National Natural Science Foundation of China (Grant No.\, 12371031).

\subsection*{AI use declaration} During preliminary exploration of this work, ChatGPT-5.2 Thinking suggested using the discrete Fourier transform as a route to a block-triangular decomposition. The author independently developed all results, verified all proofs, wrote the manuscript, and takes full responsibility for its contents.

\end{document}